\documentclass[sn-mathphys]{JORSC-jnl}

\jyear{2023}%

\theoremstyle{thmstyleone}%
\newtheorem{theorem}{Theorem}
\newtheorem{proposition}[theorem]{Proposition}%
\newtheorem{lemma}[theorem]{Lemma}%
\newtheorem{assumption}[theorem]{Assumption}%
\newtheorem{corollary}[theorem]{Corollary}
\theoremstyle{thmstylethree}%

\usepackage{amsmath,amssymb,amsfonts}%
\usepackage{indentfirst}
\usepackage{epstopdf}
\usepackage{amsthm}%
\usepackage{mathrsfs}%
\usepackage{subfigure}
\begin{document}
\renewcommand{\qedsymbol}{}
\title[Accelerated Proximal Gradient Method with Backtracking for MOP ]{Accelerated Proximal Gradient Method with Backtracking for Multiobjective Optimization }


\author[1]{Huang Chengzhi}\email{huangczmath@163.com}

\author*[2]{Chen jian}\email{chenjianmath@163.com}

\author[3]{Tang Liping}\email{tanglipings@163.com}

\affil[1]{\orgdiv{National Center for Applied Mathematics in Chongqing}, \orgname{ Chongqing Normal University}, \orgaddress{ \city{Chongqing}, \postcode{401331}, \country{China}}}

\affil[2]{\orgdiv{National Center for Applied Mathematics in Chongqing}, \orgname{ Chongqing Normal University}, \orgaddress{ \city{Chongqing}, \postcode{401331}, \country{China}}}

\affil[3]{\orgdiv{National Center for Applied Mathematics in Chongqing}, \orgname{ Chongqing Normal University}, \orgaddress{ \city{Chongqing}, \postcode{401331}, \country{China}}}
\equalcont{This work was supported by }


\abstract{This paper proposes a new backtracking strategy based on the FISTA accelerated algorithm for multiobjective optimization problems. The strategy focuses on solving the problem of Lipschitz constant being unknown. It allows estimate parameter updates non-increasingly. Furthermore, the proposed strategy effectively avoids the limitation in convergence proofs arising from the non-negativity of the auxiliary sequence, thus providing a theoretical guarantee for its performance. We demonstrate that, under relatively mild assumptions, the algorithm achieves the convergence rate of $O(1/k2)$.}



\pacs[Mathematics Subject Classification]{XXXXX}

\maketitle

\section{Introduction}\label{sec1}
In practical application scenarios, it is common to encounter problems that involve optimizing multiple objective functions simultaneously. These problems are generally referred to as multiobjective optimization problems, which typically do not have a single optimal value but rather a solution set composed of Pareto-optimal solutions. Consequently, solving such problems presents numerous challenges.

In this paper, we focus on the following so-called composite unconstrained multiobjective optimization problem:
\begin{equation}\label{mop}
	(MOP) \quad \min_{x \in \mathbb{R}^n} F(x) \equiv (f_1(x)+g_1(x), \dotsb, f_m(x)+g_m(x))^T,
\end{equation}
where $f_i :\mathbb{R}^n \to \mathbb{R}$ and $g : \mathbb{R}^n \to \mathbb{R},\forall i = 1,\dotsb,m$ are both convex functions and only $f_i$ is required to be smooth. Many practical problems, such as image restoration and compressed sensing, can be transformed into this form.

Now, the algorithms solving multiobjective problems conclude the scalarization, evolution, and gradient methods. The scalarization method suffered from the selection of weighted parameters, and the evolution method lacked the theory of its convergence. So, researchers turned to focus on the rest. Many algorithms were proposed, such as the proximal point\cite{bonnel2005proximal}, trust region\cite{carrizo2016trust}, steepest descent\cite{fliege2000steepest}, Newton\cite{fliege2009newton},  projected gradient\cite{fukuda2013inexact}, and conjugate gradient methods\cite{lucambio2018nonlinear} and proximal gradient. The proximal gradient method is useful for solving the composite multiobjective problem but has a slow convergence rate. It caused researchers to combine it with accelerated tricks to propose fast convergence rate proximal gradient methods.

Building on the accelerated method with $O(1/k^2)$ convergence rate proposed by Nesterov in \cite{nesterov1983accelerated}, numerous accelerated algorithms have been developed for solving single-objective optimization problems  (e.g., \cite{beck2009fista}, \cite{nesterov2013introductory}, \cite{nesterov2005smooth}, \cite{tseng2008acclerated}). In particular, The FISTA algorithm, introduced by Beck et al., achieves the same convergence rate for the case when $m = 1$ in problem (\ref{mop}).

 Inspired by the success of the FISTA algorithm for single-objective case \cite{beck2009fista}, Tanabe et al. extended it to the multiobjective case (mFISTA) \cite{tanabe2023accelerated}. Like FISTA, mFISTA enjoys the convergence rate $O(1/k^2)$ and computes an $\varepsilon$- solution in $O(\sqrt{L(f)/\varepsilon})$ steps, where $L(f)$ is a bound on the Lipschitz constant for $\nabla f(x)$. 

However, mFISTA requires the pre-determined constant $\ell > L(f)$, and in practical applications, the Lipschitz constant of the objective function may not be available, making it challenging to determine $\ell$. Even when using line search techniques, like those in \cite{beck2009fista}, to estimate $L(f)$, there is no guarantee of convergence because the auxiliary term $\sigma(k)$ used in \cite{tanabe2023accelerated} does not have non-negativity due to the nature of multiobjective problems. Besides, this strategy makes the Lipschitz constant's estimated value very big after some iteration, which can substantially limit the performance of FISTA since this causes the sizes of the steps taken at that point, and at all subsequent iterates, to be very small. As a result, the scaling technique used by Beck in the proof cannot be applied here. This drawback led us to reconsider whether we could reconstruct the sequences in the algorithm to bypass the difficulties caused by the lack of non-negativity in the auxiliary term, thus enabling a convergence proof. 

Fortunately, Scheinberg K et.al.\cite{scheiberg2014fast} proposed a new $\{t_k\}$ to overcome this failure. They introduced a  parameter $\theta_{k} = \mu_k / \mu_{k+1}$, where $\mu_k$ satisfied $\mu_{k} < 1/L(f)$, to define a sequence $t_{k+1} = (1 + \sqrt{ 1 + 4 \theta_k t_k^2})/{2}$ and construct an equation-like relationship between $t_k$ and $L(f)$. Then, they used a constant $\mu_{k+1}^0$ to control the size of $\mu_{k+1}$, which allowed $\mu_{k+1}$ non-increasing and avoided the small size of the step taken at the iteration point. Inspired by this, we use an iterative format like $\{t_k\}$ in \cite{scheiberg2014fast}, combined with the auxiliary item $\sigma(k)$ in \cite{tanabe2023accelerated} and the merit function $u_0(x)$ to prove the convergence rate of the algorithm. Moreover we use the same method as \cite{scheiberg2014fast} to estimate $L(f)$ to solve the iteration step size problem.Our contribution in this paper is as follows.
\begin{itemize}
	\item By referring to sequence $\{t_k\}$ in \cite{scheiberg2014fast}, we generalize it to the case of multiobjective optimization, and give its equality relationship with $\{L_k\}$. Combined with the auxiliary sequence and evaluation function proposed in \cite{tanabe2023accelerated}, we give the convergence proof of the algorithm and analyze its convergence rate.
	
	\item In addition, we also give a new line search criterion for multiobjective optimization, and on this basis, we propose a backtracking technique for estimating the Lipschitz constant of the objective functions in multiobjective optimization, and solve the problem that the step size of FISTA algorithm is too small due to too large value of Lipschitz constant in multiobjective optimization.
\end{itemize}

The remainder of this paper is organized as follows. Section \ref{Preliminary Results} presents some notations, definitions
and auxiliary results which will be used in the sequel. Section \ref{Algorithm} introduces the subproblem of our algorithm and proposes the algorithm. Section \ref{Convergence Rate} build some properties of $\{t_k\}$ and $\{W_k\}$ obtain the convergence results of algorithm. Section \ref{Numerical experiment} is devoted to the
numerical experiments. Finally, Section \ref{Conclusion} draws the conclusions of this paper.

\section{Preliminary Results}\label{Preliminary Results}
In this paper, for any $n \in \mathbb{N}$, $\mathbb{R}^{n}$ denotes the $n$-dimensional Euclidean space. And  $\mathbb{R}^{n}_{+}:= \{ v \in \mathbb{R}^{n} \mid v_{i} \geq 0, i=1,2,\dotsb,n \} \subseteq \mathbb{R}^{n}$ signify the non-negative orthant of $\mathbb{R}^{n}$. Besides, $\Delta^{n} := \{ \lambda \in \mathbb{R}^{n}_{+} \mid \lambda_{i} \geq 0, \sum_{i=1}^{n} \lambda_{i} = 1 \}$ represents the standard simplex in $\mathbb{R}^{n}$. Unless otherwise specified, all inner products in this article are taken in the Euclidean space. Subsequently, the partial orders induced by $\mathbb{R}^{n}_{+}$ are considered, where for any $v^{1}, v^{2} \in \mathbb{R}^{n}$, $v^{1} \leq v^{2}$ (alternatively, $v^{1} \geq v^{2}$) holds if $v^{2} - v^{1} \in \mathbb{R}^{n}_{+}$, and $v^{1} < v^{2}$ (alternatively, $v^{1} > v^{2}$) if $v^{2} - v^{1} \in \text{int} \, \mathbb{R}^{n}_{+}$. In case of misunderstand, we define the order $\preceq(\prec)$ in $\mathbb{R}^{n}$ as $$u\preceq(\prec)v~\Leftrightarrow~v-u\in\mathbb{R}^{n}_{+}(\mathbb{R}^{n}_{++}).$$
Instead, $u \npreceq v$.

From the so-called descent lemma [Proposition A.24 \cite{bertsekas1999nonlinear}], we
have $f_{i}(p)-f_{i}(q)\leq\langle\nabla f_{i}(q),p-q\rangle+(L/2)\|p-q\|^{2}$ for all $p,q\in\mathbb{R}^{n}$ and $i=1,\ldots,m$, which gives

\begin{equation}\label{descent lemma}
	\begin{aligned}F_{i}(p)-F_{i}(r)&=f_{i}(p)-f_{i}(q)+g_{i}(p)+f_{i}(q)-F_{i}(r)\\&\leq\langle\nabla f_{i}(q),p-q\rangle+g_{i}(p)+f_{i}(q)-F_{i}(r)+\frac{L}{2}\left\|p-q\right\|^{2}
	\end{aligned}
\end{equation}
for all $p,q,r\in\mathbb{R}^{n}$ and $i=1,\ldots,m.$

To construct the proximal gradient algorithm, we first introduce some basic definitions for the upcoming discussion. For a closed, proper, and convex function $ h: \mathbb{R}^n \to \mathbb{R} \cup \{ \infty \} $, the Moreau envelope of $ h $ is defined as

$$
\mathcal{M}_h(x) := \min_{y \in \mathbb{R}^n} \left\{ h(y) + \frac{1}{2} \left\| x - y \right\|^2 \right\}.
$$

The unique solution to this problem is called the proximal operator of $ h $, denoted as

$$
\operatorname{prox}_h(x) := \arg\min_{y \in \mathbb{R}^n} \left\{ h(y) + \frac{1}{2} \left\| x - y \right\|^2 \right\}.
$$

Next, we introduce a property between the Moreau envelope and proximal operation by following the lemma.
\begin{lemma}[\cite{rockafellar1997convex}]
	If  {$h$} is a proper closed and convex function, the Moreau envelope
	$\mathcal{M}_{h}$ is Lipschitz continuous and takes the following form,
	$$\nabla \mathcal{M}_{h}(x) := x- prox_{h}(x).$$
\end{lemma}

The following assumption is made throughout the paper:

\begin{assumption}
	$f:\mathbb{R}^n \to \mathbb{R}$ is continuously differentiable with Lipschitz continuous gradient $L(f)$:
	\begin{equation}\label{assump2.1}
		\parallel \nabla f(x) - \nabla f(y) \parallel \leq L(f) \parallel x - y \parallel, \quad \forall x,y \in \mathbb{R}^n,
	\end{equation}
	where $\parallel \cdot \parallel$ stands for standard Euclidean norm unless specified otherwise.
\end{assumption}

We now revisit the optimality criteria for the multiobjective optimization problem denoted as (\ref{mop}). An element $x^{*} \in \mathbb{R}^n$ is deemed weakly Pareto optimal if there does not exist $x \in \mathbb{R}^{n}$ such that $F(x) < F(x^{*})$, where $F: \mathbb{R}^{n} \to \mathbb{R}^{m}$ represents the vector-valued objective function, The ensemble of weakly Pareto optimal solutions is denoted as $X^{*}$. The merit function $u_{0}: \mathbb{R}^{n} \to \mathbb{R} \cup \{\infty\}$, as introduced in \cite{tanabe2023new}, is expressed in the following manner:

\begin{equation}\label{3}
	u_0(x) := \sup_{z \in \mathbb{R}^{n}} \min_{i 	=1,\dotsb,m}[F_{i}(x) - F_{i}(z)].
\end{equation}

The following lemma proves that $u_0$ is a merit function in the Pareto sense.

\begin{lemma}[\cite{tanabe2023accelerated}]
	Let $u_0$ be given as (\ref{3}), then $u_0(x) \geq 0, x \in \mathbb{R}^{n}$, and $x$ is the
	weakly Pareto optimal for (\ref{mop}) if and only if $u_0(x) = 0$.
\end{lemma}

\section{ An accelerated proximal gradient method for multiobjective optimization}\label{Algorithm}
This section presents an accelerated version of the proximal gradient method. Similar to \cite{tanabe2023accelerated}, we start from (\ref{descent lemma}) and give the subproblem required for each iteration: for a given $x \in dom {F}$, $y \in \mathbb{R}^n$:

\begin{equation}\label{subproblem}
	\min_{z \in \mathbb{R}^n} \varphi_{L(f)} (z; x, y),
\end{equation}
where
$$
\varphi_{L(f)} (z; x, y) := \max_{i=1,\dotsb,n} [ \langle 	\nabla f_i(y), z - y\rangle + g_i(z) + f_i(y) - F_i(x)] + \frac{L(f)}{2} \parallel z - y \parallel^2.
$$

Since $g_i$ is convex for all $i = 1, \dotsb, m$, $z \mapsto \varphi_{L(f)} (z; x, y)$ is strong convex. Thus, the subproblem (\ref{subproblem}) has a unique optimal solution $p_{L(f)}(x, y)$ and takes the optimal function value $\theta_{L(f)} (x, y)$, i.e.,

\begin{equation}\label{pl and theta}
	\begin{aligned}
		p_{L(f)}(x, y) &:= {\arg\min}_{z \in \mathbb{R}^n} 	\varphi_{L(f)}(z; x, y) \\
		\theta_{L(f)}(x, y) &:= {\min}_{z \in \mathbb{R}^n} \varphi_{L(f)}(z; x, y). 	
	\end{aligned}
\end{equation}

Moreover, the optimality condition of (\ref{subproblem}) implies that for all $x \in dom {F}$ and $y \in \mathbb{R}^n$ there exists $\tilde{g}(x, y) \in \partial g(p_{L(f)}(x, y))$ and a Lagrange multiplier $\lambda(x, y) \in \mathbb{R}^m$ such that
\begin{equation}\label{optimal condition a}
	\sum_{i=1}^{m} \lambda_i(x, y) [ \nabla f_i(y) + 	\tilde{g}_i(x, y)] = - L(f) [p_{L(f)}(x, y) - y]
\end{equation}

\begin{equation}\label{optimal condition b}
	\lambda(x, y) \in \Delta^m, \quad \lambda_j (x, y) = 0 	\quad for \ all \ j \notin \mathcal{I}(x, y),
\end{equation}
where $\Delta^m$ denotes the standard simplex and
\begin{equation}
	\mathcal{I}(x, y) := {\arg\min}_{i = 1, \dotsb, m} [ 	\langle \nabla f_i(x), p_{L(f)}(x, y) - y\rangle + g_i(p_{L(f)} (x, y)) + f_i(y) - F_i(x)].
\end{equation}

We note that by taking $z = y$ in the objective function of (\ref{subproblem}), we get
\begin{equation}
	\theta_{L(f)}  (x, y) \leq \varphi_{L(f)} (y; x, y) = 	\max_{i=1,\dotsb,m} \{ F_i(y) - F_i(x)\}
\end{equation}
for all $x \in dom {F}$ and $y \in \mathbb{R}^n$. Moreover, from (8) with $p = z, q = y, r = x$, it follows that
$$
\theta_{L(f)} (x, y) \geq \max_{i=1,\dotsb,m} \{ F_i(p_{L(f)}(x, y)) - F_i(x)\}
$$
for all $x \in dom {F}$ and $y \in \mathbb{R}^n$. Now we use the same statement of [Proposition 4.1, \cite{tanabe2023accelerated}] to characterize weak Pareto optimality in terms of the mappings $p_{L(f)}$ and $\theta_{L(f)}$.

\begin{proposition}\label{prop for pl}
	Let $p_{L(f)}(x, y)$ and $\theta_{L(f)}(x, y)$ be defined by (\ref{pl and theta}). Then, the following three conditions are equivalent:
	
	(i) $y \in \mathbb{R}^n$ is weakly Pareto optimal for (\ref{mop});
	
	(ii) $p_{L(f)} (x, y) = y$ for some $x \in \mathbb{R}^n$;
	
	(iii) $\theta_{L(f)} (x, y) = \max_{i=1,\dotsb,m} [F_i(y) - F_i(x)]$ for some $x \in \mathbb{R}$.
\end{proposition}

\begin{proof}
	The proof is similar to [Proposition 4.1, \cite{tanabe2023accelerated}], so we omit it.
\end{proof}

Proposition \ref{prop for pl} suggests that we can use $\parallel p_{L(f)}(x, y) - y \parallel_{\infty} < \varepsilon$ for some $\varepsilon > 0$ as a stopping criterion. For convenience, we will denote the following iterative format as FISTA-Step:
\begin{equation}\label{fista step}
	\begin{cases}
		t_{k+1} &:= \frac{1 + \sqrt{ 1 + 4 \omega_k t_k^2}}{2} \\
		\theta_{k+1} &:= \frac{t_k - 1}{t_{k+1}} \\
		y_{k+1} &:= x_k + \theta_{k+1} [x_k - x_{k-1}].
	\end{cases}
\end{equation}
where $\omega_k := L_{k+1}/ L_{k}$ for all $k \geq 1$, $L_k$ is the estimated value of $L_{f}$ at each iteration.

Now, we state below the proposed algorithm.
\begin{algorithm}[H]
	\renewcommand{\algorithmicrequire}{\textbf{Input:}}
	\renewcommand{\algorithmicensure}{\textbf{Output:}}
	\caption{FISTA for multiobjective optimization with new line search technique}
	\label{alg1}
	\begin{algorithmic}[1]
		\Require {Set $x_{-1} = x_0 = y_0 \in \mathbb{R}^n, L_1 > 0, \beta > 1, \sigma > 1, \epsilon >0, t_{-1} = 0, \omega_{-1} = 1, k \leftarrow 0,t_0 \leftarrow 1$ and a vector of dimension $m$ with all components equal to one: $e_m$.}
		\For{$k = 1, \dotsb, K$}
		\State{ Set $L_k \leftarrow L_{k-1} / \sigma$.}
		\State Compute $\nabla f_i(y_k), \forall i = 1,\dotsb,m$, and compute $\hat{x}_k = p_{L_k}(x_k, y_k)$.
		\If{$F(\hat{x}_k) \npreceq  \langle J F(y_k),  \hat{x}_k - y_k\rangle 
			+ F(y_k) 
			+  \frac{L_k}{2} \parallel \hat{x}_k - y_k \parallel^2 e_m$}
		\State $L_k \leftarrow \beta L_k, \omega_{k-1} \leftarrow \beta\omega_{k-1} $
		\State $(t_k, y_k) \leftarrow \text{FISTA-Step}(x_{k-1}, x_{k-2}, t_{k-1}, \omega_{k-1})$, return to Step 3.
		\Else
		\State $x_{k+1} \leftarrow \hat{x}_k$.
		\EndIf
		\If{$\parallel p_{L(f)}(x_k, y_k) - y_k \parallel_{\infty} < \varepsilon$}
		\Return {Weak Pareto solution $x^*$} 
		\EndIf
		\EndFor
	\end{algorithmic}
\end{algorithm}
Algorithm \ref{alg1} is known to generate $x_k$ such that $u_0(x_k)$ converges to zero with rate $O(1/k^2)$ under the following assumption.

\begin{assumption}\label{a1}
	Suppose $ X^{*} $ is set of the weakly Pareto optimal points and $\mathcal{L}_{F}(c) := \{x \in \mathbb{R}^{n} F(x) \leq c\}$, then for any $x \in \mathcal{L}_{F}(F(x_0)), k\geq 0$, then there exists
	$x \in X^{*}$such that $F(x^{*}) \leq F(x)$ 
	$$
	\begin{aligned}
		R := \sup_{F^{*} \in F(X^{*} \cap \mathcal{L}_{F}( F(x_0))} &\inf_{z \in F^{-1}(\{F^{*}\})}
		\parallel x_0 - z \parallel^2 < \infty.
	\end{aligned}
	$$
\end{assumption}

Now we give some properties of  $\{t_k\}$ and $\{\theta_k\}$ to facilitate subsequent convergence proofs.

\begin{theorem}\label{th of tk}
	Let $\{t_k\}$ be defined by (\ref{fista step}). Then, the following properties hold for all $k \geq 1$:
	
	(i) $\frac{1}{2} + \sqrt{\omega_{k}}t_k \leq t_{k+1} \leq \left( 1 + \sqrt{\omega_{k}} \right) t_k$;
	
	(ii) $L_{k+1}^{-1} t_{k+1} (t_{k+1} - 1) = L_k^{-1} t_k^2$;
	
	
	
	(iii) For any $k \in \mathbb{N}$, $\theta_{k+1}^2 \leq \frac{L_k}{L_{k+1}} < 1$;
\end{theorem}

\begin{proof}
	These properties can be derived simply from the definition of $\{t_k\}$.
\end{proof}

\section{Convergence Rate}\label{Convergence Rate}
This section focuses on proving that the algorithm has a convergence rate of $O(1/k^2)$ under Assumption 1. Before doing so, let us introduce some auxiliary sequences in \cite{tanabe2023accelerated} to prove the convergence rate. For $k \geq 0$, let $\sigma_k : \mathbb{R}^n \to \mathbb{R}\cup \{- \infty\}$ and $u_k : \mathbb{R}^n \to \mathbb{R}$ be similarly defined in the following form
\begin{equation}\label{aux seq}
	\begin{aligned}
		\sigma_k(z) &:= \min_{i=1,\dotsb,m} [F_i(x_k) - F_i(z)] \\
		\rho_k(z) &:= t_{k-1}x_k - (t_{k-1} - 1)x_{k-1} - z.
	\end{aligned}
\end{equation}

Now, we give an important property of $\sigma_k(z)$ that will be helpful in the following discussions.

\begin{proposition}\label{prop5.2}
	
	Let $\sigma_k(z)$ define as (\ref{aux seq}), we have
	\begin{equation}\label{15}
		\sigma_k(z) - \sigma_{k+1}(z)
		\geq -\frac{L_{k }}{2} [2\langle  y_k -x_{k+1} , y_k -x_k\rangle  + \parallel x_{k+1} - y_k \parallel^2]. ,
	\end{equation}
	and
	\begin{equation}\label{16}
		\begin{aligned}
			\sigma_{k+1}(z) \leq& \frac{L_{k }}{2}[2\langle y_k-x_{k+1},y_k-z\rangle -\|x_{k+1}-y_k\|^2].
		\end{aligned}
	\end{equation}
\end{proposition}

\begin{proof}
	Recall that there exists $\tilde{g}(x_k, y_k) \in \partial g(x_{k+1})$ and Lagrange multiplier $\lambda_i(x_k, y_k) \in R^m$ that satisfy the KKT condition (\ref{optimal condition a}, \ref{optimal condition b}) for the subproblem (\ref{subproblem}), we have
	$$
	\begin{aligned}
		\sigma_{k+1}(z)   = \min_{i = 1,\dotsb,m} [F_i(x_{k+1}) -  F_i(z)] \leq \sum_{i=1}^{m} \lambda_i(x_k, y_k) [F_i(x_{k+1}) -  F_i(z) ].
	\end{aligned}
	$$
	Essential from the descent lemma,
	$$
	\begin{aligned}
		\sigma_{k+1}(z)  &\leq \sum_{i=1}^{m} \lambda_i(x_{k}, y_k) [F_i(x_{k+1}) -  F_i(z) ] \\
		&\leq \sum_{i=1}^{m} \lambda_i(x_{k}, y_k) [\langle \nabla f_i(y_k ), x_{k+1} - y_k \rangle + g_i(x_{k+1})+ f_i(y_k ) - F_i(z )  ] \\
		&\quad+ \frac{L_{k }}{2} \parallel x_{k+1} - y_k \parallel^2.
	\end{aligned}
	$$
	Hence, the convexity of $f_i$ and $g_i$ yields
	$$
	\begin{aligned}
		\sigma_{k+1}(z)  &\leq \sum_{i=1}^{m} \lambda_i(x_{k}, y_k) [\langle \nabla f_i(y_k ) + \tilde{g}(x_k, y_k), x_{k+1} - z \rangle  ] + \frac{L_k}{2} \parallel x_{k+1} - y_k \parallel^2.
	\end{aligned}
	$$
	Using (\ref{optimal condition a}) with $x = x_k$ and $y = y_k$ and from the fact that $x_{k+1} = p_\ell(x_k, y_k)$, we obtain
	$$
	\begin{aligned}
		\sigma_{k+1}(z)
		&\leq \frac{L_{k } }{2}[2\langle y_k - x_{k+1}, y_k -z \rangle - \parallel x_{k+1} - y_k \parallel^2],
	\end{aligned}
	$$
	
	From the definition of $\sigma_k(z)$, we obtain
	$$
	\sigma_k(z) - \sigma_{k+1}(z)
	\geq - \max_{i =1, \dotsb,m}[F_i(x_{k+1}) - F_i(x_k)] 	
	$$
	Essential from the descent lemma,
	$$
	\begin{aligned}
		&\quad \sigma_k(z)  - \sigma_{k+1}(z)  \\
		&\geq - \max_{i =1, \dotsb,m} [\langle \nabla f_i(y_k), x_{k+1} - y_k \rangle + g_i(x_{k+1}) + f_i(y_k) -F_i(x_k)]- \frac{L_{k}}{2} \parallel x_{k+1} - y_k \parallel^2	\\
		&= -\sum_{i=1}^{m} \lambda_i(x_k, y_k) [\langle \nabla f_i(y_k), x_{k+1} - y_k \rangle + g_i(x_{k+1}) + f_i(y_k) -F_i(x_k)] - \frac{L_{k}}{2} \parallel x_{k+1} - y_k \parallel^2 \\
		&= -\sum_{i=1}^{m} \lambda_i(x_k, y_k) [\langle \nabla f_i(y_k), x_{k} - y_k \rangle + f_i(y_k) - f_i(x_k)] \\
		&\quad-\sum_{i=1}^{m} \lambda_i(x_k, y_k) [\langle \nabla f_i(y_k), x_{k+1} - x_k \rangle + g_i(x_{k+1}) - g_i(x_k)] - \frac{L_{k}}{2} \parallel x_{k+1} - y_k \parallel^2.
	\end{aligned}
	$$
	where the  first equality comes from (\ref{optimal condition b}), and the second one follows by taking $x_{k+1} - y_k = (x_k - y_k) + (x_{k+1} - x_k)$. From the convexity of $f_i, g_i$, we show that
	$$
	\begin{aligned}
		\sigma_k(z)  - \sigma_{k+1}(z)  \geq& -\sum_{i=1}^{m} \lambda_i(x_k, y_k) \langle \nabla f_i(y_k) + \tilde{g}(x_k, y_k), x_{k+1} - x_{k} \rangle - \frac{L_{k}}{2} \parallel x_{k+1} - y_k \parallel^2 .	
	\end{aligned}
	$$
	Thus, (\ref{optimal condition a}) and some calculations prove that
	$$
	\begin{aligned}
		\sigma_k(z) - \sigma_{k+1}(z)
		\geq \frac{L_{k}}{2} [2\langle x_{k+1} - y_k, y_k -x_k\rangle  - \parallel x_{k+1} - y_k \parallel^2]. \\
	\end{aligned}
	$$
	So we can get two inequalities of $\sigma_{k}(z)$ and $\sigma_{k+1}(z)$ that we want.
\end{proof}

\begin{theorem}\label{fx0}
	Algorithm \ref{alg1} generates a sequence $\{x_k\}$ such that for all $i =1,\dotsb m$ and $k\geq 0$, we have
	$$F_i(x_k) \leq F_i(x_0).$$
\end{theorem}

\begin{proof}
	The proof is similar to [Theorem 5.1 \cite{tanabe2023accelerated}]; we omit it.
\end{proof}

\begin{theorem}\label{th5.4}
	Suppose $\left\{x_k\right\}$ and $\left\{y^k\right\}$ be the sequences generated by algorithm \ref{alg1}, for any $z\in\mathbb{R}^n$, it holds that
	$$
	u_{0}(x_k) \leq \frac{4\beta L(f)R}{(k+1)^2} 
	$$
\end{theorem}

\begin{proof}
	Let(\ref{15}) multiply by $(t_{k+1} - 1)$ and (\ref{16})) multiply by $t_{k+1}$, then add them, we get:
	\begin{equation}\label{reswkwk1}
		\begin{aligned}
			2L_k^{-1} t_k(t_{k} - 1)\sigma_{k}(z)  - 2L_k^{-1} t_{k}^2\sigma_{k+1}(z) &\geq  2t_{k}\langle x_{k+1}-y_{k}, t_{k} y_k - (t_{k}-1)x_k - z\rangle \\
			&\quad+t_{k}^2\|x_{k+1}-y_{k}\|^{2},
		\end{aligned}
	\end{equation}
	
	With 
	$$
	\parallel b - a \parallel^2 + 2 \langle b -a, a-c \rangle = \parallel b - c\parallel^2 - \parallel a - c \parallel^2 \quad \forall a,b,c \in \mathbb{R}^n
	$$
	and Theorem \ref{th of tk} (ii), we get
	$$
	2L_{k-1}^{-1} t_{k-1}^2\sigma_{k}(z)  - 2L_k^{-1} t_{k}^2\sigma_{k+1}(z) \geq \parallel \rho_{k+1}(z) \parallel^2 - \parallel \rho_k(z) \parallel^2 .
	$$

	Applying this inequality repeatedly, we have
	
	\begin{equation}
		2L_k^{-1} t_{k}^2\sigma_{k+1}(z) + \parallel \rho_{k+1} (z)\parallel^2 \leq 2L_0^{-1}t_{0}^2\sigma_{1}(z) + \parallel \rho_1(z) \parallel^2.
	\end{equation}
	
	And we can find that
	\begin{equation}
		2L_0^{-1} t_{0}^2\sigma_{1}(z) + \parallel \rho_1 (z)\parallel^2  = 2L_0^{-1} \sigma_{1}(z) + \parallel x_1 - z \parallel^2 \leq \parallel x_0 - z \parallel^2 .
	\end{equation}
	
	Now we must find the down bounded of $L^{-1}_k t_k^2 $. From definition of $\{t_k\}$, we get
	$$
	t_{k+1} \geq \sqrt{\frac{L_{k+1}}{L_k}}t_k + \frac{1}{2}.
	$$
	So, using the same discussion in [Lemma 3.4 \cite{scheiberg2014fast}], we get
	$$
	L^{-1}_k t_k^2 \geq \frac{k^2}{4\beta L(f)}.
	$$
	
	Hence, we find that
	
	$$
	\sigma_{k+1}(z) \leq \frac{4\beta L(f)}{(k+1)^2} \parallel x_0 - z\parallel^2
	$$
	
	With similar arguments used in the proof of Theorem 3.1 (see [41, Theorem5.2]), we get the desired inequality:
	$$
	u_{0}(x_k) \leq \frac{4\beta L(f)R}{(k+1)^2} 
	$$
\end{proof}

We conclude this section by demonstrating that the global convergence of Algorithm \ref{alg1} concerning weak Pareto optimality is also assured, leveraging the aforementioned complexity result.

\begin{corollary}[\cite{tanabe2023accelerated} Corollary 5.2]
	Suppose that Assumption \ref{assump2.1} holds. Then, every accumulation point of the sequence $\{ x_k\}$ generated by Algorithm \ref{alg1} is weakly Pareto optimal for (\ref{mop}). In particular, if the level set $\mathcal{L}_F(F(x^0))$ is bounded, then $\{x^{k}\}$ has accumulation points, and they are all weakly Pareto optimal. 
	
	Moreover, if each $F_{i}$ is strictly convex, then the accumulation points are Pareto optimum, i.e., there are no points with the same or smaller objective function values and at least one objective function value being strictly smaller.
\end{corollary}

\section{Numerical Experiment}\label{Numerical experiment}
In this section,we present numerical results to demonstrate the performance of NBPGMO for various problems. All numerical experiments  were implemented in Python3.10 and executed on a personal computer with a 12th Gen Intel(R) Core(TM) i7-12700H CPU @ 2.70 GHz and 16GB RAM. For NBPGMO, we set $L_0 = 1, \beta = 2$ and the stopping criterion $\parallel x_k - y_k \parallel_{\infty} \leq 10^{-3}$ for all tested algorithms to ensure that the algorithms terminate after a finite number of iterations. We also set the maximum number of iterations to 1000.

For convenience, we list all the test problems in Table 1 as follows and use \textbf{mFISTA\_backtracking} to represent Algorithm \ref{alg1}. Then, we use the \textbf{mFISTA} algorithm in \cite{tanabe2023accelerated} to compare the performance. The second and third columns present the dimensions of variables and objective functions, respectively. The lower and upper bounds of variables are denoted by $\mathbf{x}_L$ and $\mathbf{x}_U$, respectively. For each problem, we perform 200 computations using the same initial points for different tested algorithms. The initial points are randomly selected within the specified lower and upper bounds.

\begin{table}[htbp]
	\centering
	\caption{Test Problems}\label{test_problem}
	\begin{tabular}{llllll}
		\toprule
		Problem & $n$ & $m$ & $\mathbf{x}_L$ & $\mathbf{x}_U$ & Reference \\
		\midrule
		BK1 \& $\ell_1$ & 2 & 2 & $(-5,-5)$ & $(10,10)$ & \cite{huband2006rebview} \\
		Deb & 2 & 2 & $(0.1, 0.1)$ & $(1, 1)$ & \cite{huband2006rebview} \\
		DD1 & 5 & 2 & $(-20, \dotsb, -20)$ & $(20, \dotsb, 20)$ & \cite{dd1998normal} \\
		FF1 & 2 & 2 & $(-1,-1)$ & $(1,1)$ & \cite{huband2006rebview} \\
		Hil1 & 2 & 2 & $(0,0)$ & $(1,1)$ & \cite{hil2001generalized} \\
		Imb2 & 2 & 2 & $(-2, -2)$ & $(2,2)$ & \cite{chen2023bb} \\
		JOS1 \& $\ell_1$ & 2 & 2 & $(-5,-5)$ & $(5,5)$ & \cite{jin2001dynamic} \\
		MHHM1 & 1 & 3 & $0$ & $1$ & \cite{huband2006rebview} \\
		MHHM2 & 2 & 3 & $(0,0)$ & $(1,1)$ & \cite{huband2006rebview} \\
		PNR & 2 & 2 & $(-2,-2)$ & $(2,2)$ & \cite{preuss2006pareto} \\
		SP1 \& $\ell_1$ & 2 & 2 & $(2,-2)$ & $(3,3)$ & \cite{huband2006rebview} \\
		VFM1 & 2 & 3 & $(-2,-2)$ & $(2,2)$ & \cite{huband2006rebview} \\
		VU1 \& $\ell_1$ & 2 & 2 & $(-3,-3)$ & $(3,3)$ & \cite{huband2006rebview} \\
		WIT1 & 2 & 2 & $(-2,-2)$ & $(2,2)$ & \cite{witting2012numerical} \\
		WIT2 & 2 & 2 & $(-2,-2)$ & $(2,2)$ & \cite{witting2012numerical} \\
		WIT3 & 2 & 2 & $(-2,-2)$ & $(2,2)$ & \cite{witting2012numerical} \\
		WIT4 & 2 & 2 & $(-2,-2)$ & $(2,2)$ & \cite{witting2012numerical} \\
		WIT5 & 2 & 2 & $(-2,-2)$ & $(2,2)$ & \cite{witting2012numerical} \\
		WIT6 & 2 & 2 & $(-2,-2)$ & $(2,2)$ & \cite{witting2012numerical} \\
		\bottomrule
	\end{tabular}
	\label{tab0}
\end{table}

The recorded averages from the 200 runs include the number of iterations, CPU time, and purity results of the obtained Pareto fronts.


\begin{table}[htbp] 
	\centering
	\caption{Performance Comparison of NBPGMO and APGMO}
	\begin{tabular}{lccc|ccc}
		\toprule
		Problem & \multicolumn{3}{c}{NBPGMO} & \multicolumn{3}{c}{APGMO} \\
		\cmidrule(r){2-4} \cmidrule(l){5-7}
		& Iter & Time & Purity & Iter & Time & Purity \\
		\midrule
		BK1 \& $\ell_1$   & \textbf{2745}  & \textbf{4.39}  & \textbf{0.42}  & 9001   & 39.97 & 0.28 \\
		DD1    & \textbf{8633}  & \textbf{11.36} & \textbf{0.84}  & 19719  & 13.45 & 0.83 \\
		Deb    & \textbf{4425}  & \textbf{5.66}  & \textbf{0.96}  & 7146   & 6.09  & 0.94 \\
		FF1    & \textbf{2627}  & \textbf{3.40}  & \textbf{0.79}  & 36835  & 41.20 & 0.64 \\
		Hil1   & \textbf{1910}  & \textbf{4.51}  & 0.41  & 9466   & 12.64 & \textbf{0.57} \\
		Imb2   & \textbf{6438}  & \textbf{8.19}  & \textbf{0.77}  & 138364 & 92.50 & 0.26 \\
		JOS1 \& $\ell_1$   & \textbf{5324}  & \textbf{6.66}  & 0.40  & 15035  & 72.27 & \textbf{0.90} \\
		MHHM1  & \textbf{1324}  & \textbf{1.66}  & \textbf{0.81}  & 1786   & 1.79  & \textbf{0.81} \\
		MHHM2  & \textbf{1493}  & \textbf{1.88}  & \textbf{0.89}  & 2322   & 2.01  & \textbf{0.89} \\
		PNR    & \textbf{1310}  & \textbf{2.13}  & 0.40  & 3529   & 2.80  & \textbf{0.67} \\
		SP1	\& $\ell_1$   & \textbf{1959}	& \textbf{3.19}  & \textbf{0.95}  & 9253	&55.01 	&0.73 \\ 			
		VFM1   & \textbf{1412}  & \textbf{0.99}  & \textbf{0.53}  & 26022  & 27.80 & \textbf{0.53} \\
		VU1 \& $\ell_1$   & \textbf{12429}	& \textbf{15.93} & \textbf{0.82}  & 149583	&286.36 	&0.33 		\\		
		WIT1   & \textbf{45644} & \textbf{59.23} & 0.91  & 58168  & 61.85 & \textbf{0.95} \\
		WIT2   & \textbf{22408} & \textbf{30.05} & \textbf{0.73}  & 60802  & 46.06 & 0.71 \\
		WIT3   & \textbf{17278} & \textbf{22.58} & \textbf{0.77}  & 41060  & 31.61 & 0.73 \\
		WIT4   & \textbf{5584}  & \textbf{8.44}  & \textbf{0.70}  & 20768  & 15.60 & 0.60 \\
		WIT5   & \textbf{4777}  & \textbf{7.00}  & \textbf{0.88}  & 6694   & 7.72  & 0.51 \\
		WIT6   & \textbf{1105}  & \textbf{1.71}  & \textbf{0.91}  & 2105   & 1.82  & 0.90 \\
		\bottomrule
	\end{tabular}
	\label{tab1}
\end{table}

The obtained Pareto sets and fronts for some test problems are depicted in Figures \ref{f1}.

\begin{figure}[H]
	\centering
	\subfigure[BK1 \& $\ell_1$]
	{
		\begin{minipage}[H]{.2\linewidth}
			\centering
			\includegraphics[scale=0.1]{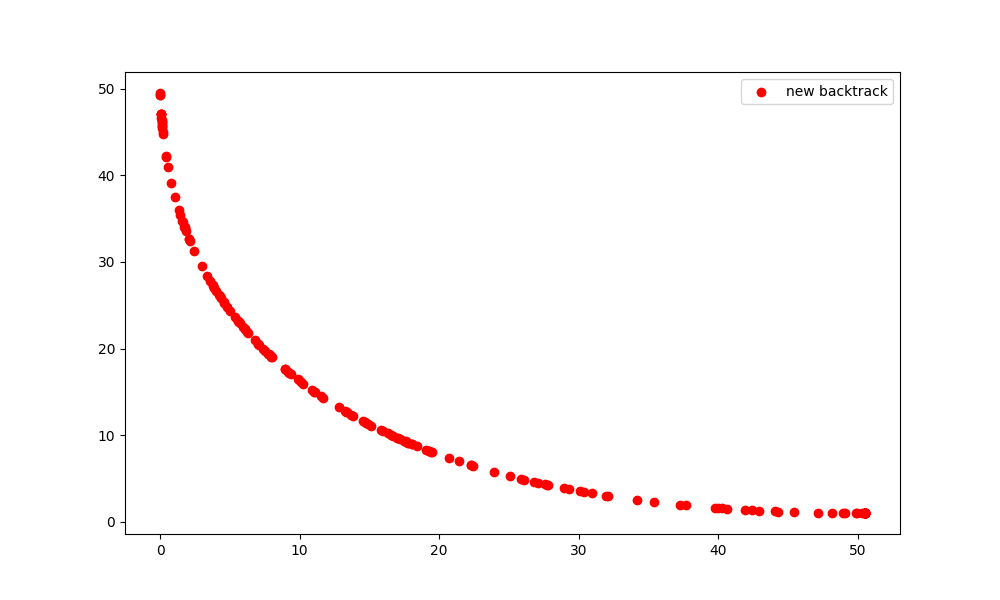} \\
			\includegraphics[scale=0.1]{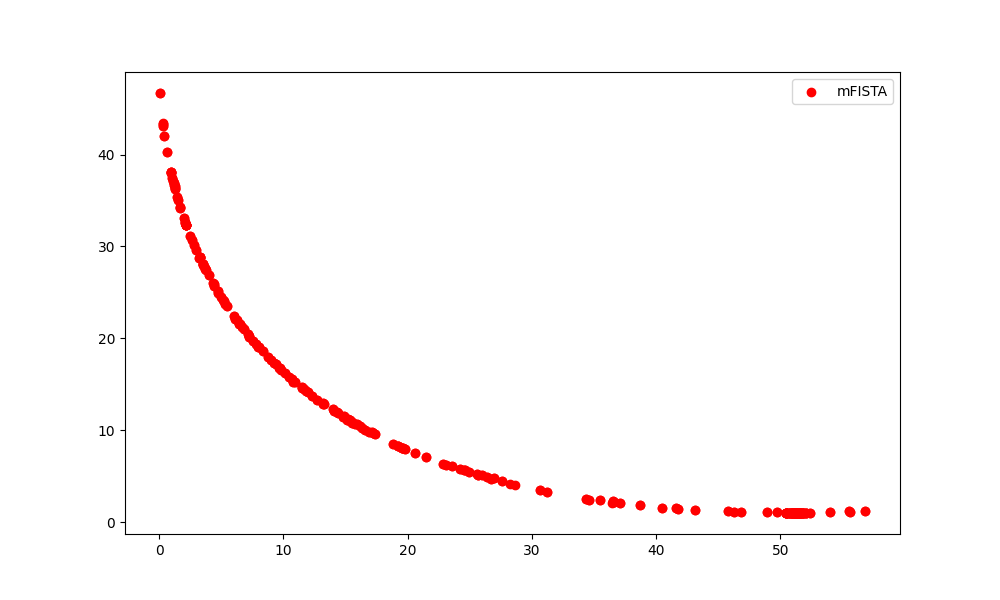}
		\end{minipage}
	}
	\subfigure[JOS1 \& $\ell_1$]
	{
		\begin{minipage}[H]{.2\linewidth}
			\centering
			\includegraphics[scale=0.1]{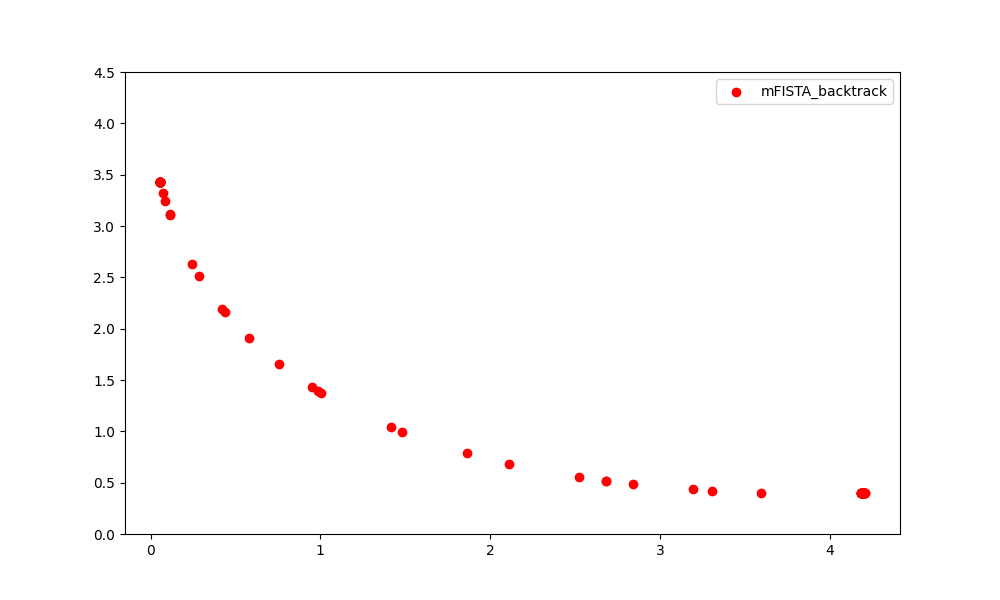} \\
			\includegraphics[scale=0.1]{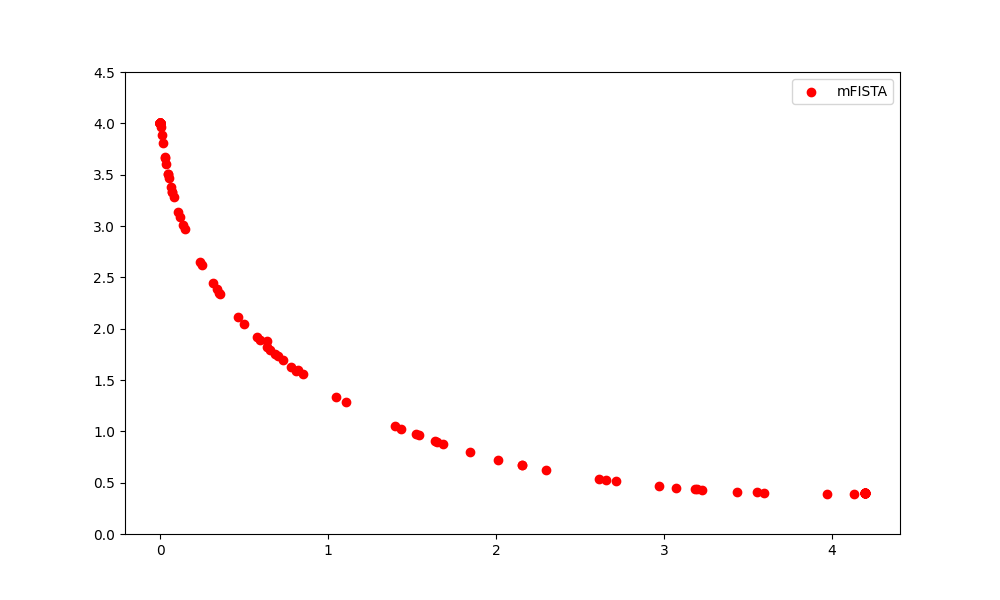}
		\end{minipage}
	}
	\subfigure[SP1 \& $\ell_1$]
	{
		\begin{minipage}[H]{.2\linewidth}
			\centering
			\includegraphics[scale=0.1]{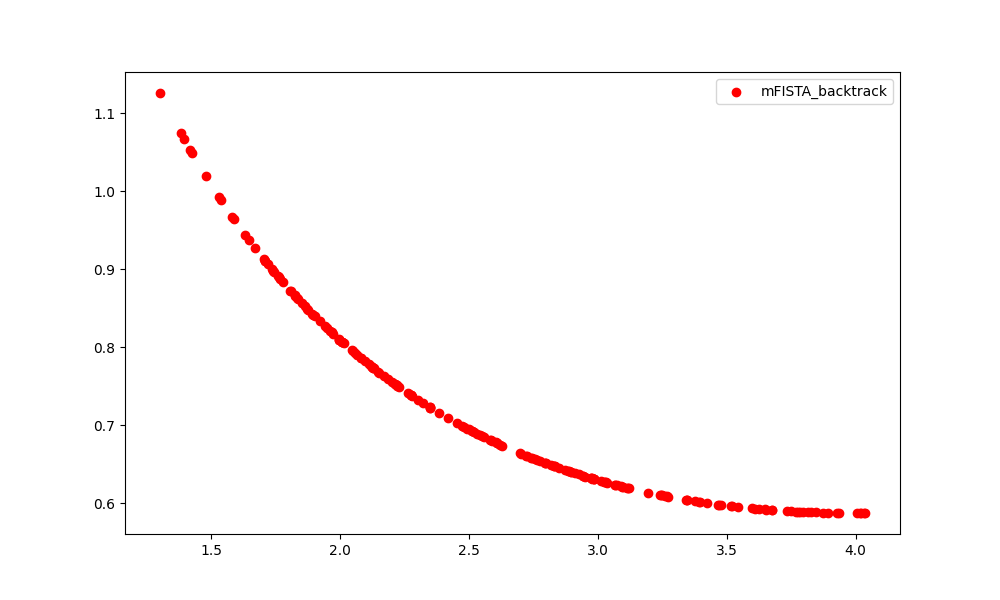} \\
			\includegraphics[scale=0.1]{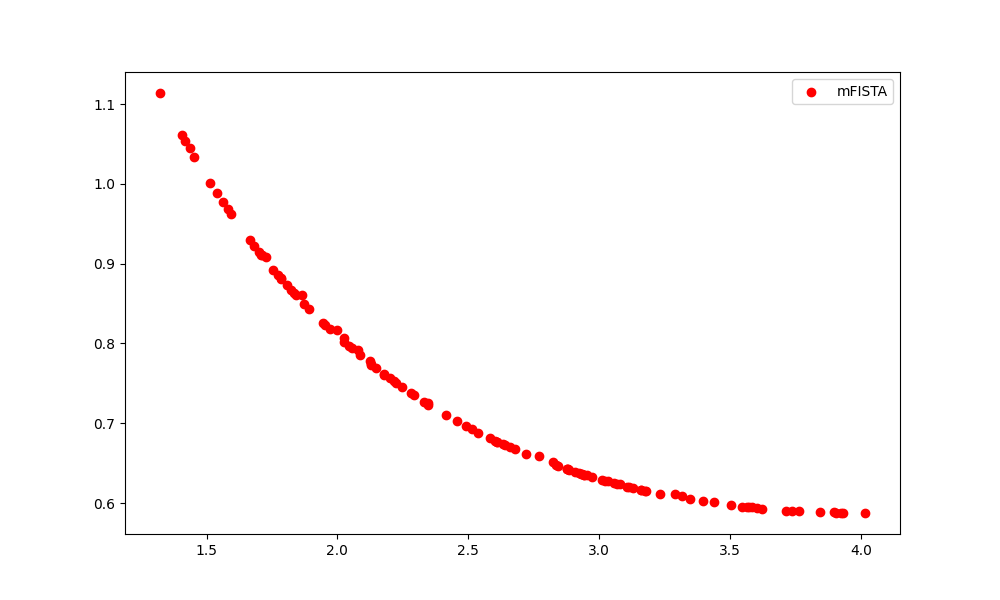}
		\end{minipage}
	}
	\subfigure[VU1 \& $\ell_1$]
	{
		\begin{minipage}[H]{.2\linewidth}
			\centering
			\includegraphics[scale=0.1]{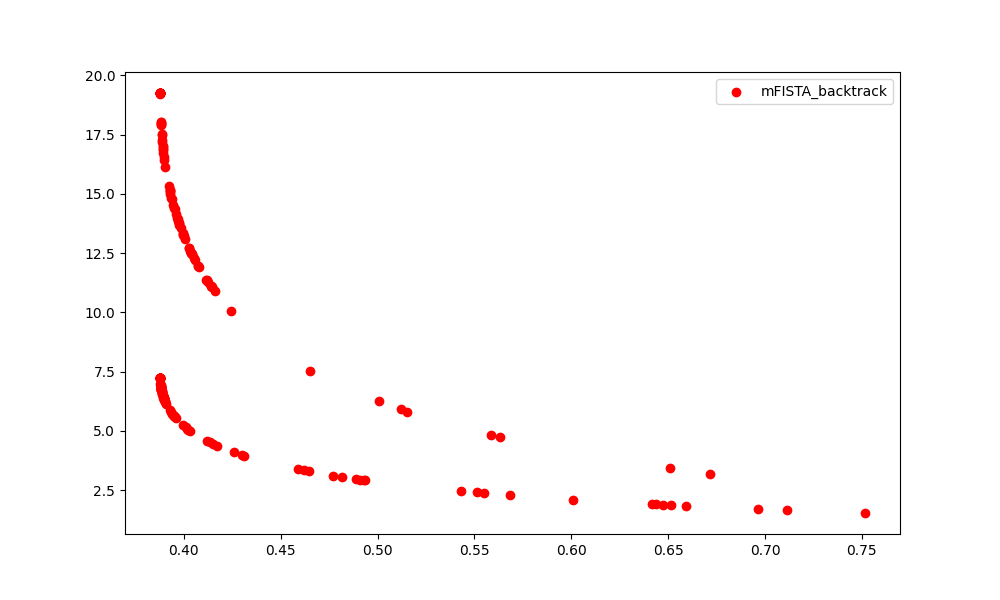} \\
			\includegraphics[scale=0.1]{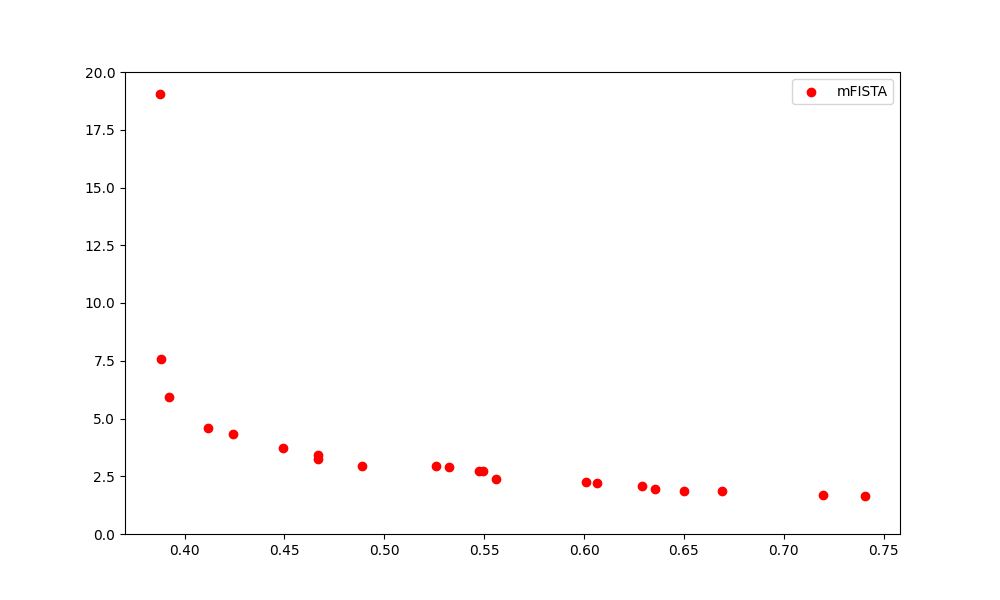}
		\end{minipage}
	}

\subfigure[FF1]
{
	\begin{minipage}[H]{.2\linewidth}
		\centering
		\includegraphics[scale=0.1]{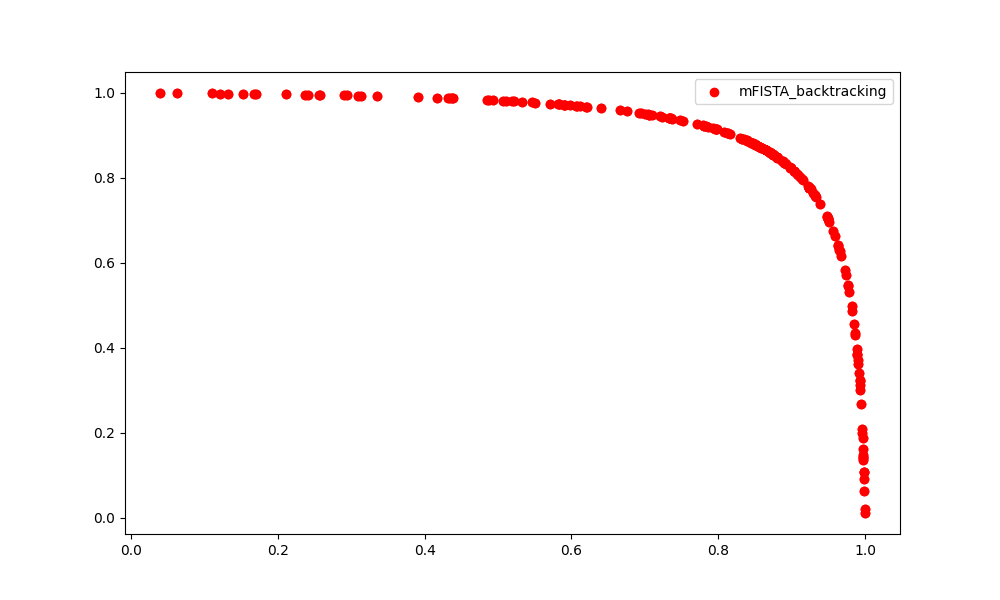} \\
		\includegraphics[scale=0.1]{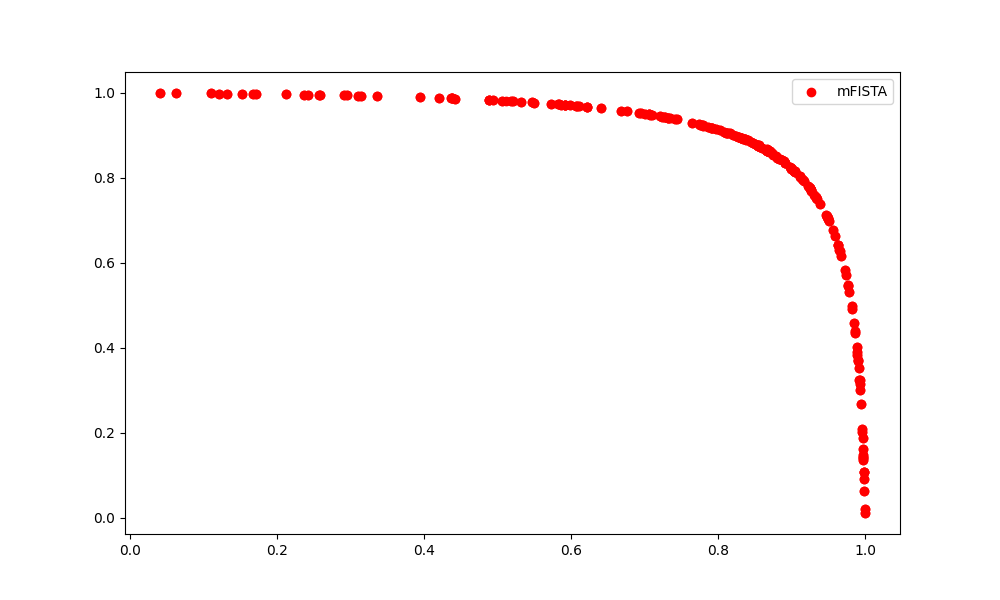}
	\end{minipage}
}
\subfigure[Hil1]
{
	\begin{minipage}[H]{.2\linewidth}
		\centering
		\includegraphics[scale=0.1]{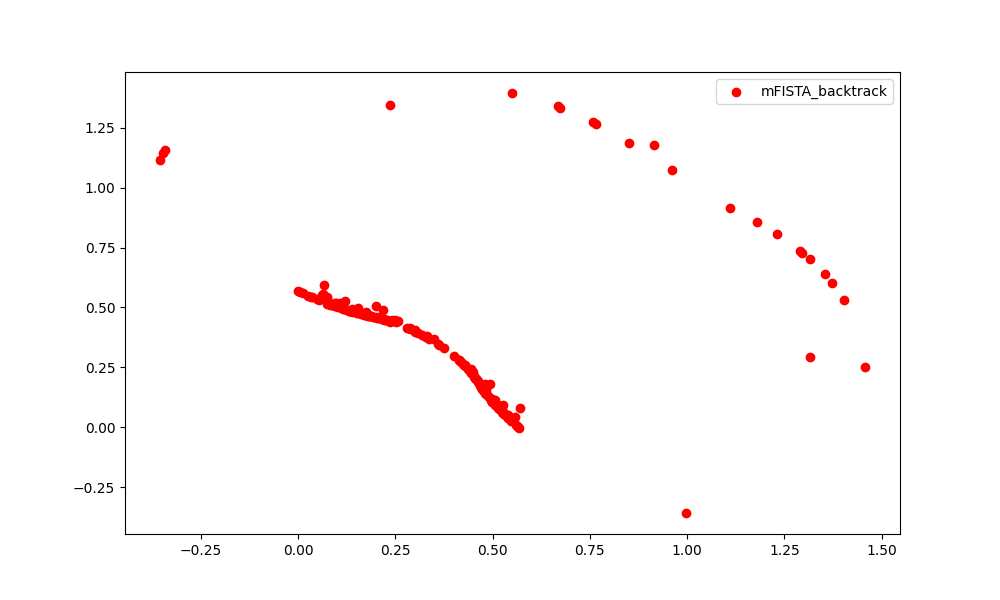} \\
		\includegraphics[scale=0.1]{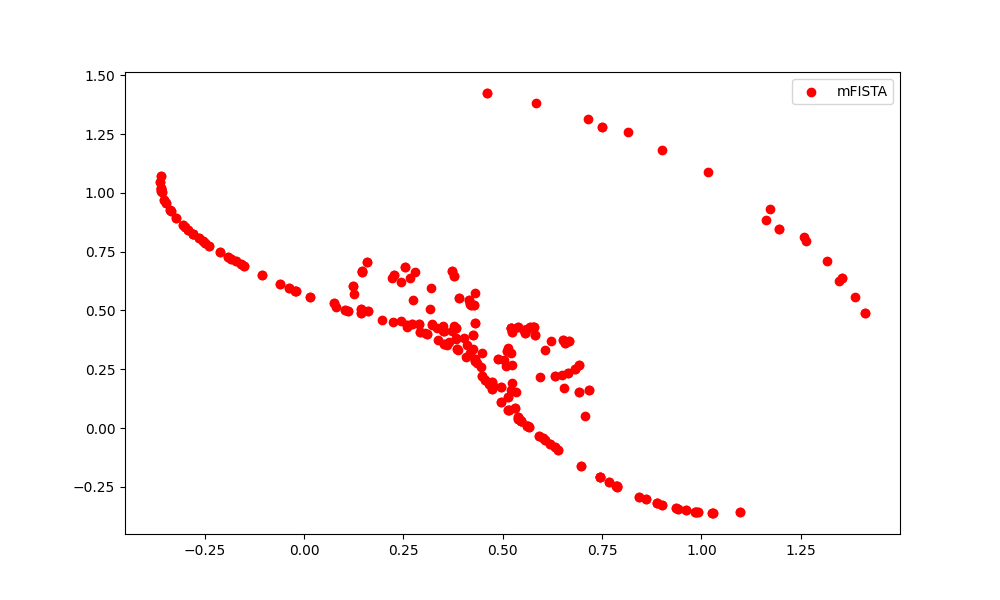}
	\end{minipage}
}
\subfigure[PNR]
{
	\begin{minipage}[H]{.2\linewidth}
		\centering
		\includegraphics[scale=0.1]{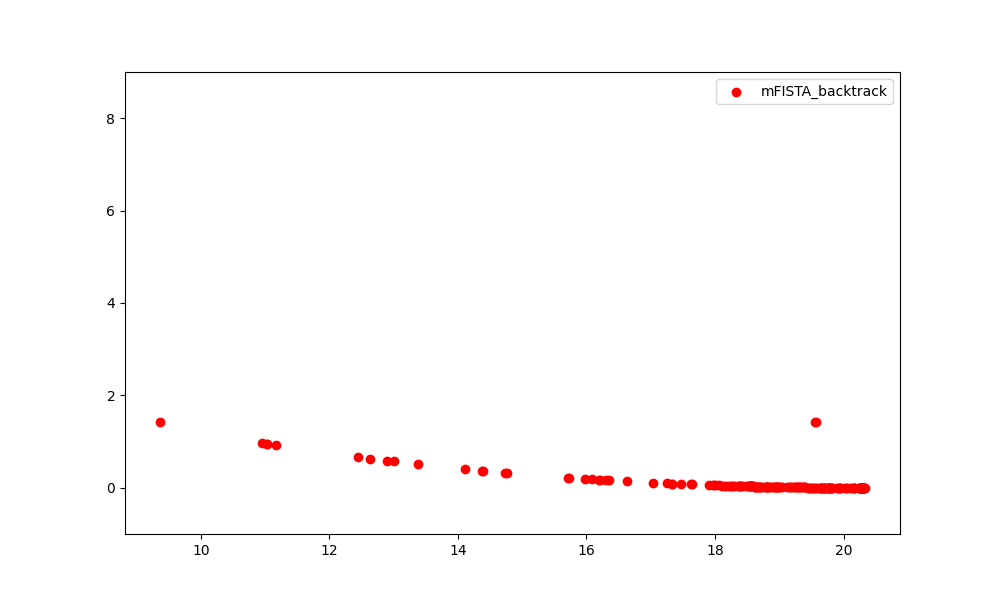} \\
		\includegraphics[scale=0.1]{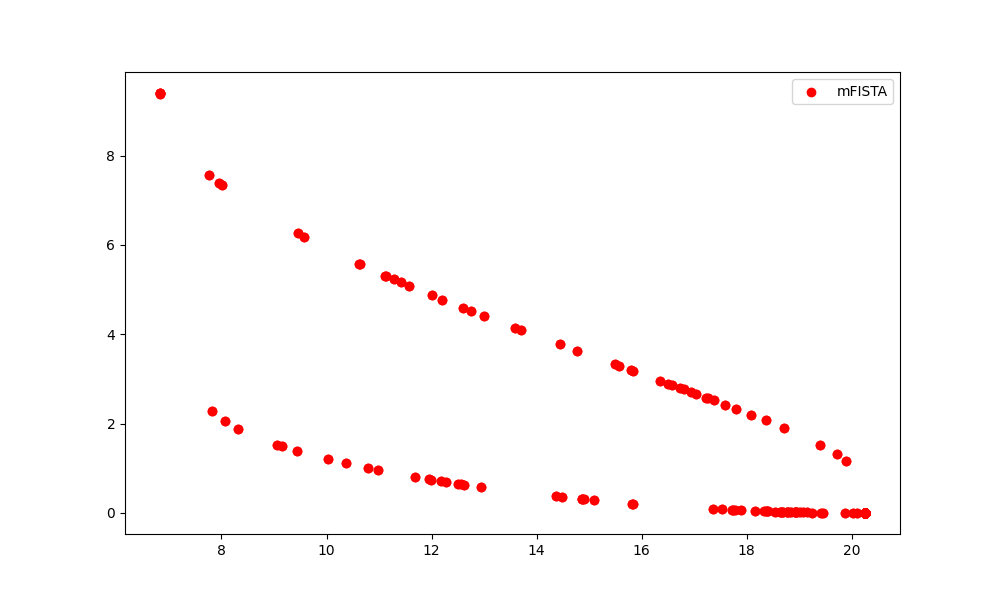}
	\end{minipage}
}
\subfigure[DD1]
{
	\begin{minipage}[H]{.2\linewidth}
		\centering
		\includegraphics[scale=0.1]{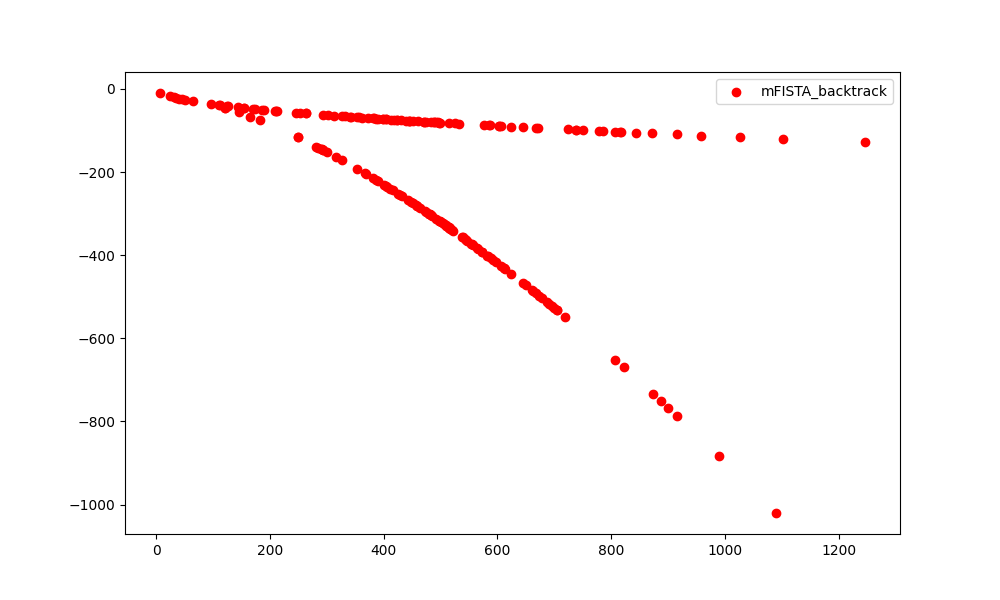} \\
		\includegraphics[scale=0.1]{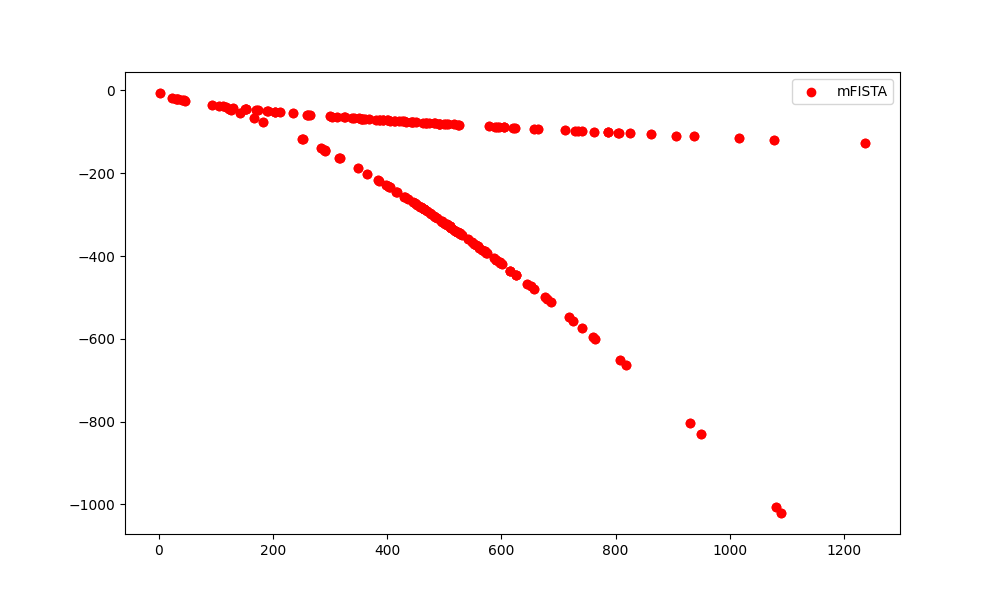}
	\end{minipage}
}

	\caption{Numerical results in variable space obtained by mFISTA\_backtracking (\textbf{Top}) and mFISTA for problems BK1 \& $\ell_1$ to DD1.}
	\label{f1}
\end{figure}

\begin{figure}[H]
	\centering	
	\subfigure[VFM1]
	{
		\begin{minipage}[H]{.3\linewidth}
			\centering
			\includegraphics[scale=0.2]{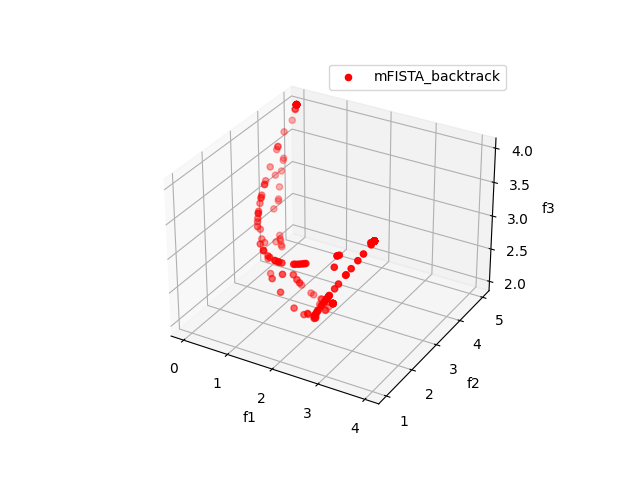} \\
			\includegraphics[scale=0.2]{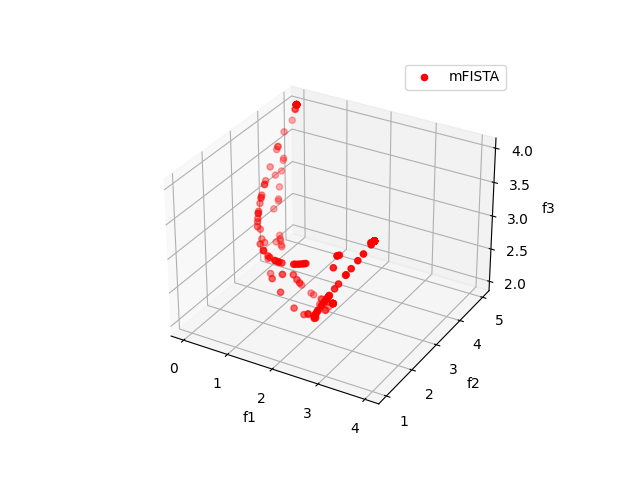}
		\end{minipage}
	}
	\subfigure[MHHM1]
	{
		\begin{minipage}[H]{.3\linewidth}
			\centering
			\includegraphics[scale=0.2]{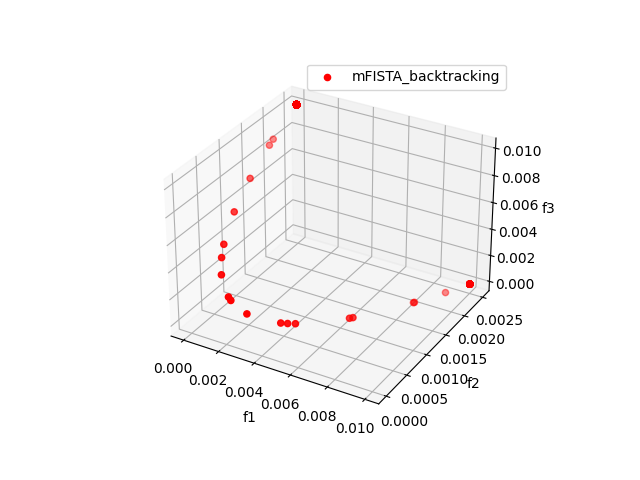} \\
			\includegraphics[scale=0.2]{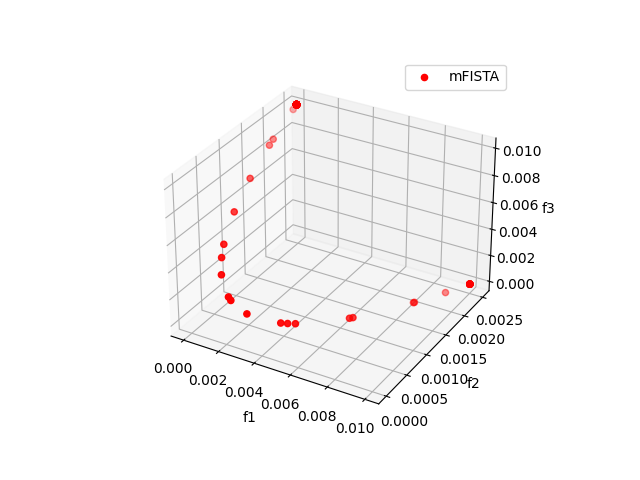}
		\end{minipage}
	}
	\subfigure[MHHM2]
	{
		\begin{minipage}[H]{.3\linewidth}
			\centering
			\includegraphics[scale=0.2]{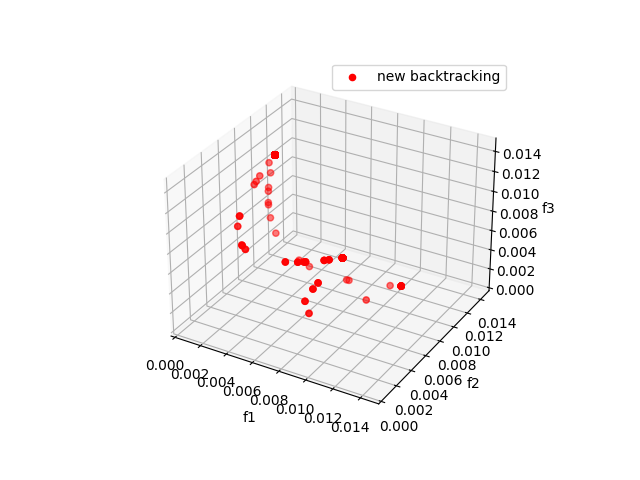} \\
			\includegraphics[scale=0.2]{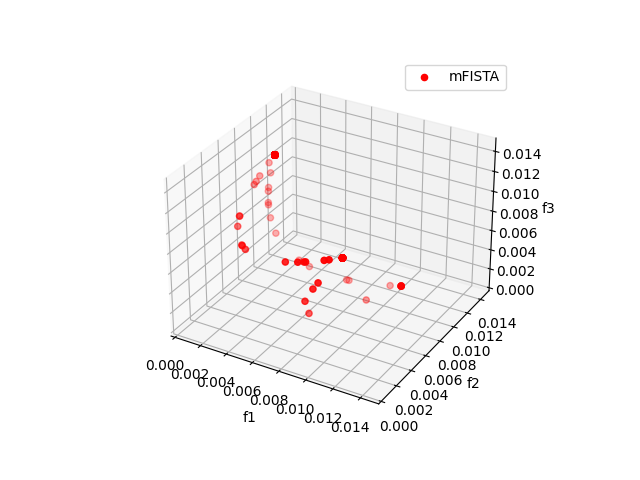}
		\end{minipage}
	}
	
	\caption{Numerical results in variable space obtained by mFISTA\_backtracking (\textbf{Top}) and mFISTA for problems VFM1, MHHM1 and MHHM2.}
	\label{f2}
\end{figure}

\begin{figure}[H]
	\centering 	
	\subfigure[mFISTA]{
		\includegraphics[width=0.45\textwidth]{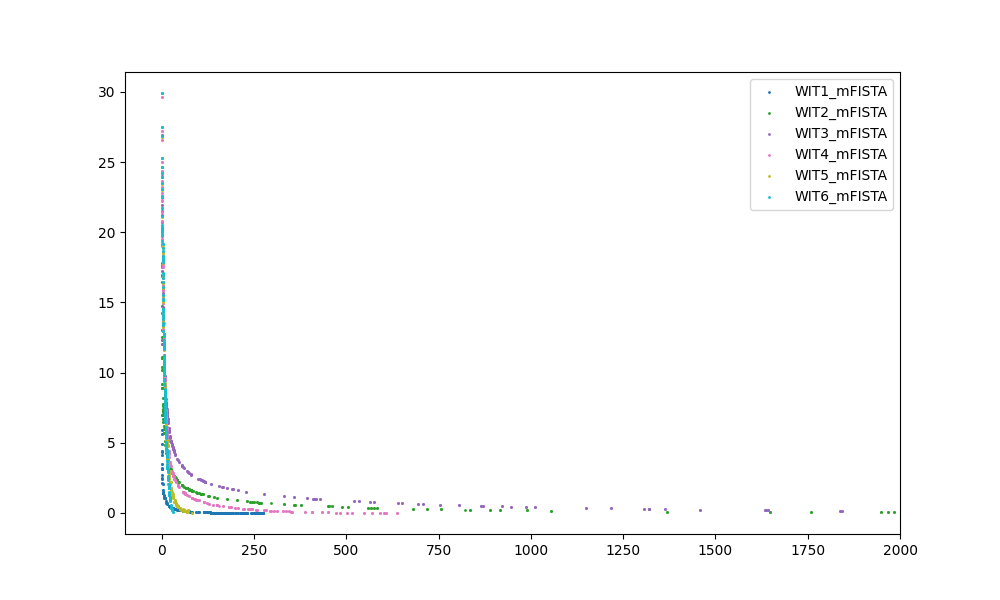}
	}
	\hfill
	\subfigure[mFISTA\_backtracking]{
		\includegraphics[width=0.45\textwidth]{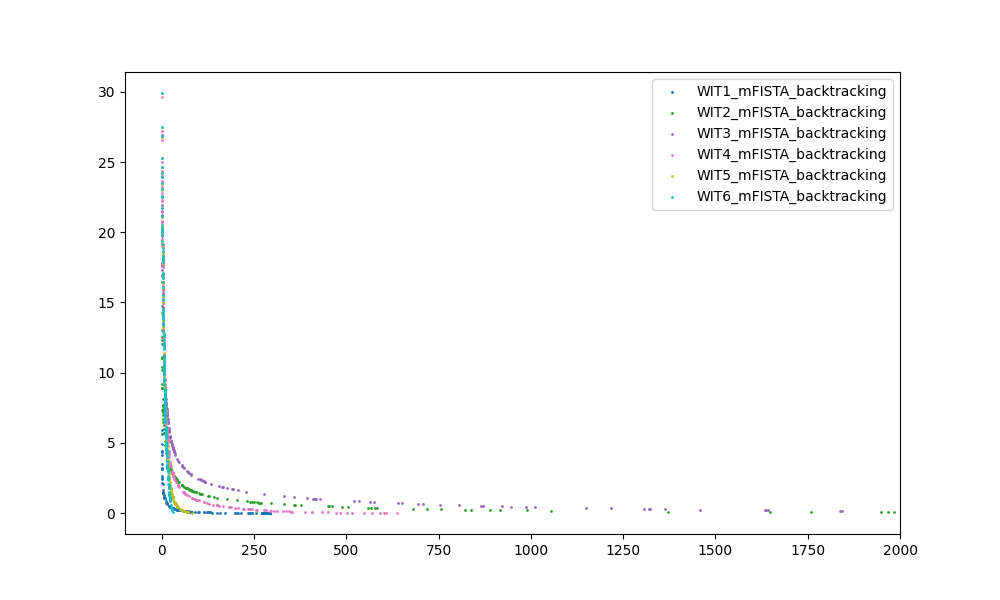}
	}
	
	\centering
	\caption{Numerical results in variable space obtained by mFISTA\_backtracking (\textbf{Left}) and mFISTA for problems WIT1 - WIT6.}
	\label{WIT}
\end{figure}

To assess the effectiveness of the obtained Pareto front, we use purity as a distinguishing metric, as shown in Table \ref{tab1}. In addition, we provide performance metric plots for time, number of iterations, and purity.

\begin{figure}[H]
	\centering 	
	\subfigure[Time]{
		\includegraphics[width=0.3\textwidth]{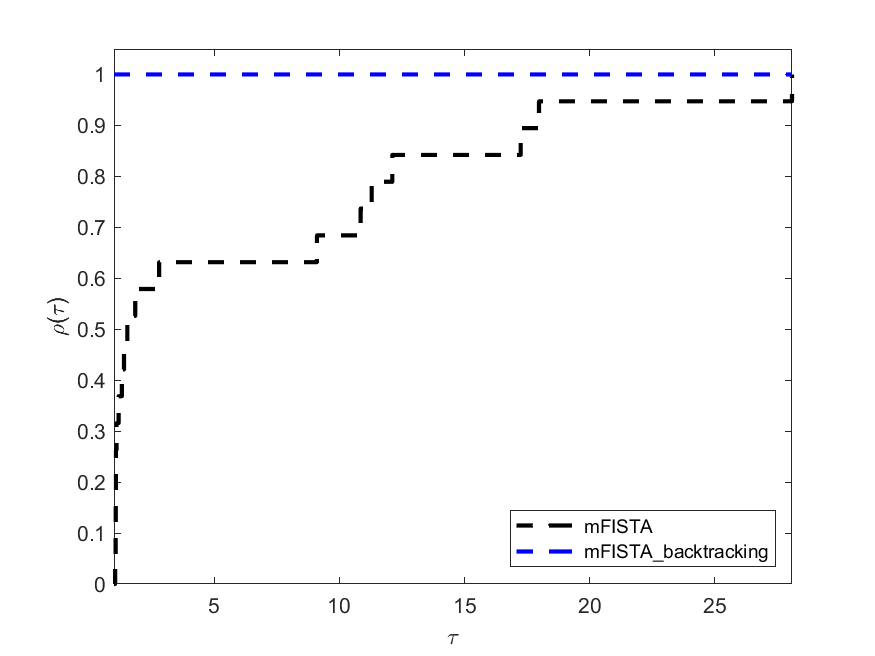}
		\label{time}
	}
	\hfill
	\subfigure[Iterations]{
		\includegraphics[width=0.3\textwidth]{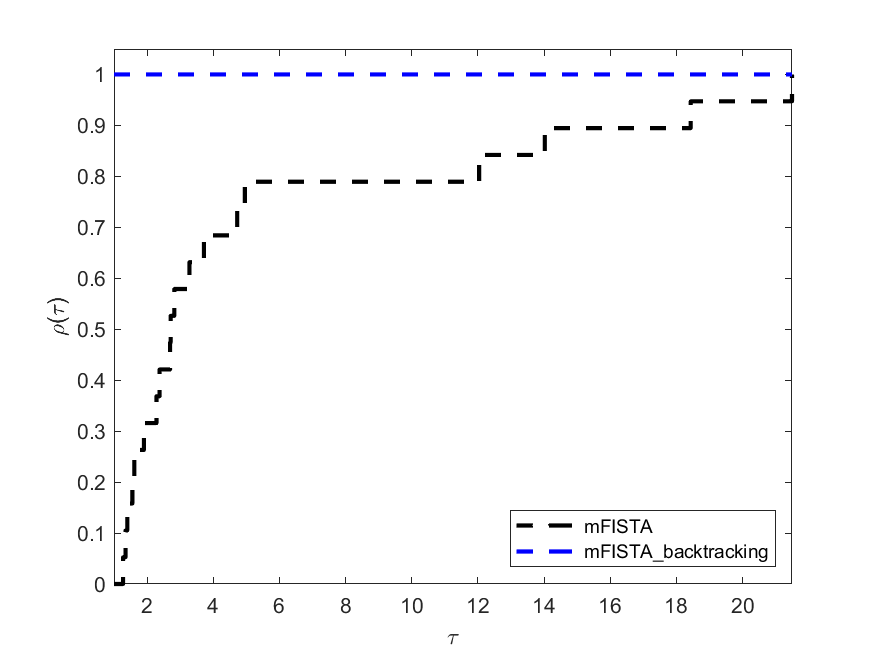}
		\label{iter}
	}
	\hfill
	\subfigure[Purity]{
		\includegraphics[width=0.3\textwidth]{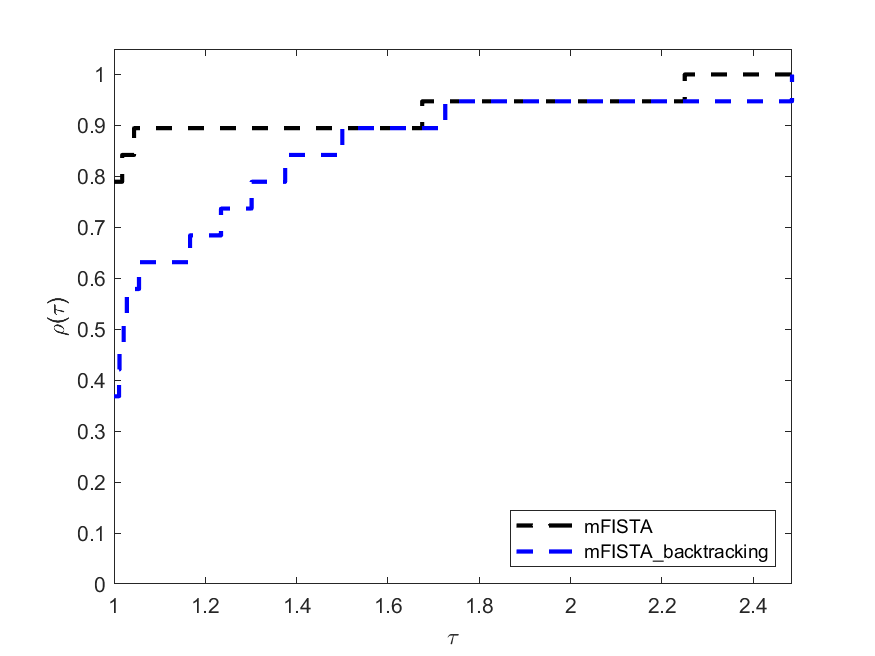}
		\label{purity}
	}
	
	\caption{Performance Metrics.}
	\label{metric}
\end{figure}

From the above results, it can be seen that the \textbf{mFISTA\_backtracking} algorithm does not have an advantage in terms of the number of iterations, as its estimation of the Lipschitz constant leads to more iterations. However, it has an advantage in running time and also outperforms in terms of the purity of the obtained Pareto front. This suggests that the estimation of the Lipschitz constant can promote the algorithm's performance, making it closer to the true Pareto front when Lipschitz constant is unknown compared to the \textbf{mFISTA} algorithm without such an estimation.

\section{Conclusion}\label{Conclusion}
We design a new backtracking technique to estimate the Lipschitz constant $L(f)$ and introduce a new sequence $\{t_k\}$ to address the challenges in the convergence proof of the traditional FISTA algorithm for multiobjective optimization problems. Specifically, we overcome the issue of the missing non-negativity of the auxiliary sequence $W_k$ and the situation where the gradient Lipschitz constant $L(f)$ of the objective function is unknown. By establishing an equality relationship between $L(f) $ and $ t_k $, we avoid the difficulties associated with the non-scalability of the missing non-negativity of $ W_k $. Furthermore, due to the backtracking technique, the sequence $ \{L_k\} $ of the estimates for $ L(f) $ does not require monotonicity, preventing tremendous values of $ L_k $ that would result in small iteration step sizes. In the future, we will further explore the performance of this new strategy in broader application scenarios, such as nonsmooth multiobjective optimization problems.

\backmatter

\end{document}